\documentclass[10pt]{amsart}
\usepackage[cp1251]{inputenc}
\usepackage[english,russian]{babel}
\usepackage{amsmath}
\usepackage{amssymb}
\usepackage{amsfonts}
\usepackage{graphicx}

\newtheorem{theorem}{Theorem}

\newtheorem{corollary}{Corollary}

\begin{document}
	\renewcommand{\refname}{References}
	\renewcommand{\proofname}{Proof.}
	\renewcommand{\figurename}{Fig.}

	\thispagestyle{empty}
	
	\title[The volume of a spherical antiprism]{The volume of a spherical antiprism}
	\author{{N.V. ABROSIMOV, B. VUONG}}%
	\address{Nikolay Abrosimov \smallskip 
		\newline\hphantom{iii} Regional Scientific and Educational Mathematical Center, 
		\newline\hphantom{iii} Tomsk State University,
		\newline\hphantom{iii} pr. Lenina, 36,
		\newline\hphantom{iii} 634050, Tomsk, Russia\smallskip
		\newline\hphantom{iii} Sobolev Institute of Mathematics,
		\newline\hphantom{iii} pr. Koptyuga, 4,
		\newline\hphantom{iii} 630090, Novosibirsk, Russia\smallskip
		\newline\hphantom{iii} Novosibirsk State University,
		\newline\hphantom{iii} Pirogova str., 2,
		\newline\hphantom{iii} 630090, Novosibirsk, Russia\smallskip}%
	\email{abrosimov@math.nsc.ru}%

	\address{Bao Vuong \smallskip 
		\newline\hphantom{iii} Regional Scientific and Educational Mathematical Center,
		\newline\hphantom{iii} Tomsk State University,
		\newline\hphantom{iii} pr. Lenina, 36,
		\newline\hphantom{iii} 634050, Tomsk, Russia\smallskip
		\newline\hphantom{iii} Novosibirsk State University,
		\newline\hphantom{iii} Pirogova str., 1,
		\newline\hphantom{iii} 630090, Novosibirsk, Russia\smallskip}%
	\email{vuonghuubao@live.com}%
	
	\thanks{\copyright \ 2021 Abrosimov N.V., Vuong B}
	\thanks{\rm This work was supported by the Ministry of Science and Higher Education of Russia (agreement No. ~075-02-2020-1479/1)}

	\maketitle {\small
		\begin{quote}
			\noindent{\sc Abstract. } We consider a spherical antiprism. It is a convex polyhedron with $2n$ vertices in the spherical space $\mathbb{S}^3$. This polyhedron has a group of symmetries $S_{2n}$ generated by a mirror-rotational symmetry of order $2n$, i.e. rotation to the angle $\pi/n$ followed by a reflection. We establish necessary and sufficient conditions for the existence of such polyhedron in $\mathbb{S}^3$. Then we find relations between its dihedral angles and edge lengths in the form of cosine rules through a property of a spherical isosceles trapezoid. Finally, we obtain an explicit integral formula for the volume of a spherical antiprism in terms of the edge lengths.\medskip
			
		\end{quote}
	}

	\section{Introduction}
	
	An {\em antiprism} $\mathcal{A}_n$ is a convex polyhedron with two equal regular $n$-gons as the top and the bottom and $2n$ equal triangles as the lateral faces. The antiprism can be regarded as a drum with triangular sides (see Fig.~1 where for $n=5$ the lateral boundary is shown).
	
	An antiprism $\mathcal{A}_n$ with $2n$ vertices has a symmetry group $S_{2n}$ generated by a mirror-rotational symmetry of order $2n$ denoted by $C_{2n\,h}$ (in Sh\"onflies notation). In Hermann--Mauguin notation this type of symmetry is denoted by $\overline{2n}$. The element $C_{2n\,h}$ is a composition of a rotation by the angle of $\pi/n$ about an axis passing through the centres of the top and the bottom faces and reflection with respect to a plane perpendicular to this axis and passing through the middles of the lateral edges (see Fig.~2).
	
	\begin{figure}[th]
		\centering
		\includegraphics[width=0.7\textwidth]{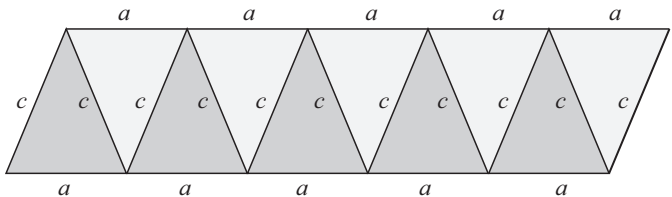}
		\caption{The lateral faces of antiprism $\mathcal{A}_5$\label{fig1}}
	\end{figure}
	
	\begin{figure}[th]
		\centering
		\includegraphics[width=0.7\textwidth]{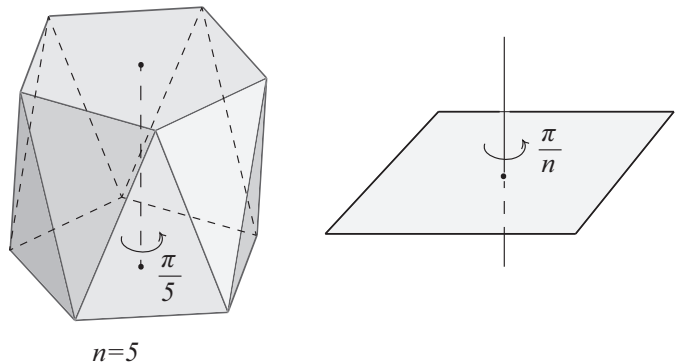}
		\caption{The symmetry of an antiprism $\mathcal{A}_5$\label{fig2}}
	\end{figure}
	
	The above definitions of an antiprism $\mathcal{A}_n$ and its symmetry group $S_{2n}$ take place for Euclidean and also for the hyperbolic and spherical space. By definition, $\mathcal{A}_n$ has two types of edges. Denote by $a$ the length of those edges that form top and bottom $n$-gonal faces. Set $c$ for the length of the lateral edges. Let $A, C$ denotes the dihedral angles respectively.
	
	The volume of a compact hyperbolic antiprism was given by the authors in \cite{hyper}. An ideal antiprism in $\mathbb{H}^3$ with all vertices at infinity was studied by A.Yu.~Vesnin and A.D.~Mednykh \cite{VesMed} (see also \cite{Ves}). A particular case of ideal rectangular antiprism is due to W.P.~Thurston \cite{Thu}. 
	
	In the present paper we consider a spherical antiprism. We also derive a formula for regular spherical tetrahedron and regular spherical octahedron as some particular cases of such family of antiprisms.
	
	\section{Regular spherical tetrahedron and regular spherical octahedron}
	
	For the sake of completeness, we bring volume formulae of the regular tetrahedron and regular octahedron with edge length $a$ in spherical space. However, these results might be known.
	
	A regular spherical tetrahedron is a special case of an antiprism. When $n=2$ the $n$-gons at the top and the bottom of an antiprism $\mathcal{A}_n$ degenerate to the corresponding two skew edges and we get a tetrahedron with symmetry group $S_4$. Hence, for $n=2$ and $a=c$ we have a regular spherical tetrahedron (see Fig.~3, left). 
	\begin{figure}[th]
		\centering
		\includegraphics[width=0.7\textwidth]{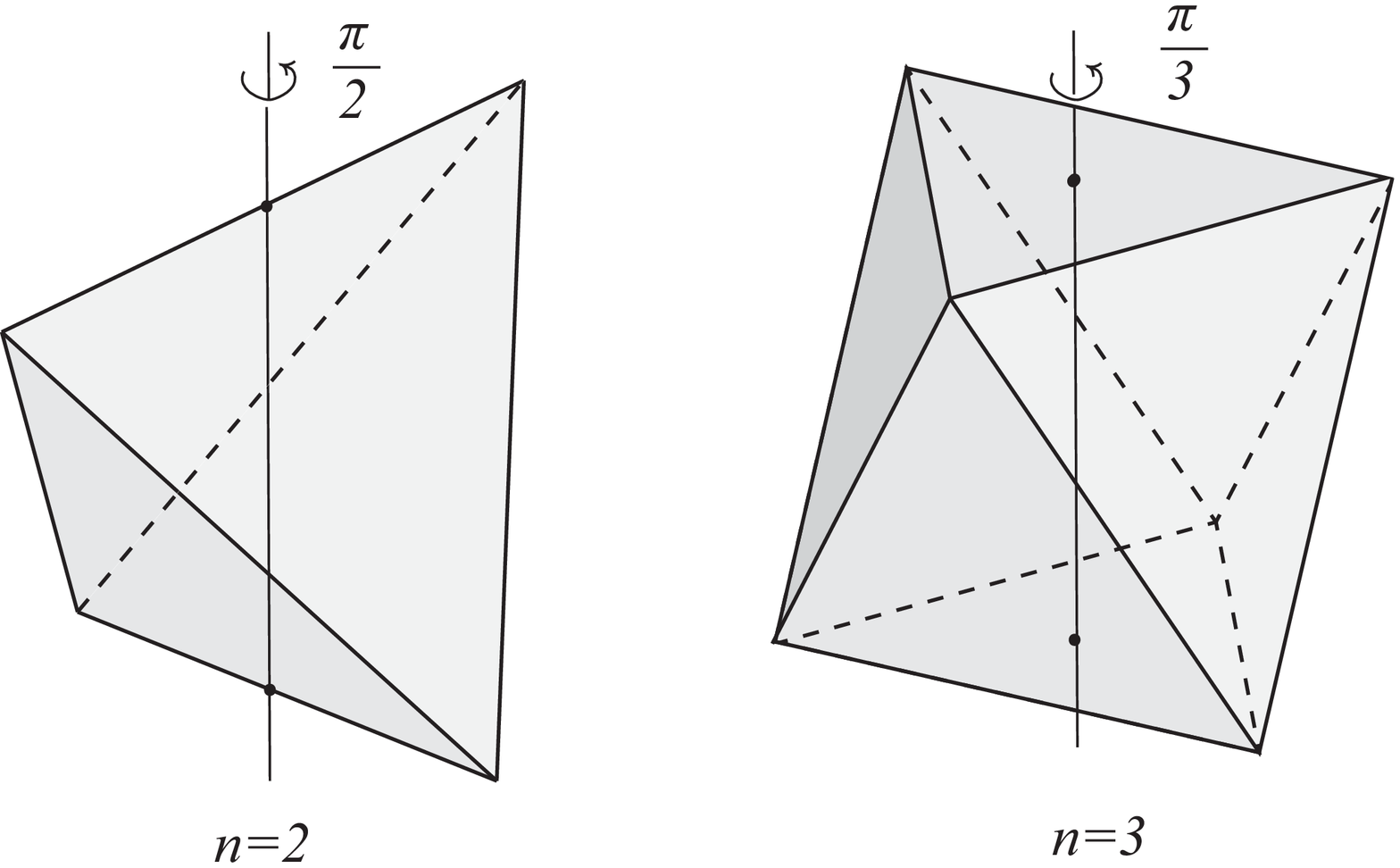}
		\caption{Particular antiprisms $\mathcal{A}_2$ and $\mathcal{A}_3$\label{fig3}}
	\end{figure}
	
	\begin{theorem}\label{regtet}
		Let $\mathcal{T}(a)$ be a regular spherical tetrahedron with edge lengths $a$ and dihedral angles $A$. Then the volume $V=vol(\mathcal{T})$ can be found by either of the following formulae
		\begin{align*}\label{Volumetetra}
		V&=3\displaystyle \int_{\arccos \frac{1}{3}}^A  \arccos\left(\frac{\cos\varphi}{1-2\cos\varphi}\right)\,d\varphi,\\
		V&=\int_{0}^a \frac{3\,t\sin t\,dt}{(1+2\cos t)\sqrt{(1+\cos t)(1+3\cos t)}}.
		\end{align*}
	\end{theorem}
	
	\begin{proof}
		Consider a spherical regular tetrahedron $\mathcal{T}$. The length of the edges is $a$, the dihedral angle is $A$. Denote the face angles by $\alpha$ as shown in Fig. \ref{stet1}, left. 
		\begin{figure}[th]
			\centering
			\includegraphics[width=0.7\textwidth]{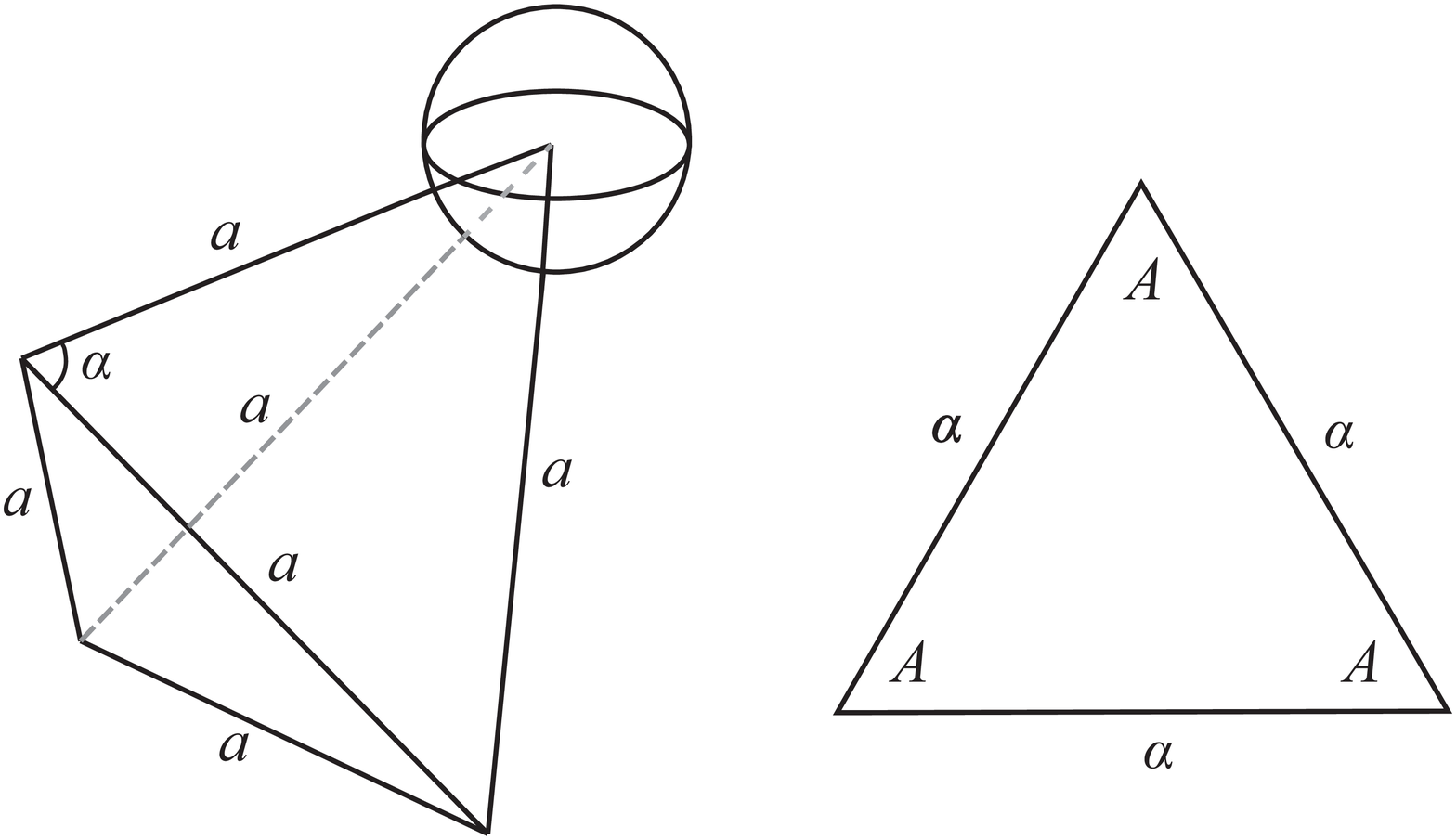}
			\caption{Regular spherical tetrahedron $\mathcal{T}$ and its intersection with a sphere centered at a vertex\label{stet1}}
		\end{figure}
		
		The faces of $\mathcal{T}$ are equal regular spherical triangles. Applying the spherical cosine rule to a face, we get the following identity
		
		\[
		\cos a = \cos ^2 a + \sin^2 a \cos \alpha
		\]
		or equivalently
		\begin{equation}\label{cosalpha}
		\cos \alpha = \frac{\cos a}{1+\cos a}.    
		\end{equation}
		
		Consider an intersection of $\mathcal{T}(a)$ with a sufficiently small sphere centered at a vertex of $\mathcal{T}$ (see Fig.~\ref{stet1}, right). Without loss of generality, we assume that the intersection is a regular spherical triangle with interior angles $A$ and sides $\alpha$. By the spherical cosine rule we have
		\begin{equation}\label{cosA}
		\cos A=\frac{\cos \alpha}{1+\cos \alpha}.
		\end{equation}
		
		Substituting (\ref{cosalpha}) into (\ref{cosA}), we get a relation between dihedral angles and edge lengths of $\mathcal{T}$
		\begin{equation}\label{1}
		\cos A=\frac{\cos a}{1+2\cos a}
		\end{equation}
		or equivalently
		\begin{equation}\label{2}
		a =\arccos \frac{\cos A}{1-2\cos A}.
		\end{equation}
		
		Observe, that when $a\to 0$ the spherical tetrahedron $\mathcal{T}$ degenerates to a single point and its volume $V\to 0$. At the same time, its dihedral angles $A\to\arccos\frac{1}{3}$.
		
		By the Schl\"afli formula (see, e.g., \cite{Vinberg}, Chap.~7, Sec.~2.2) we have the following expression for a differential of volume $V=vol(\mathcal{T})$
		\begin{equation}\label{Schl}
		d V=\sum_{\theta}\frac{\ell_{\theta}}{2}\,d\theta=3\,a\,dA,
		\end{equation}
		where the sum is taken over all the edges of $\mathcal{T}$ and $\ell_{\theta}$ is the length of an edge with dihedral angle $\theta$ along it.
		
		We substitute (\ref{2}) in (\ref{Schl}) and integrate it with respect to variable $A$ denoted by $\varphi$. Thus, by the Newton--Leibniz formula we express the volume of $\mathcal{T}$ in terms of its dihedral angles
		\begin{equation*}\label{3}
		V=3\int_{\arccos\frac{1}{3}}^{A}{\arccos}\left(\frac{\cos\varphi}{1-2\,\cos\varphi}\right)d\varphi\,.
		\end{equation*}
		
		Differentiating equation (\ref{Schl}) with respect to $a$, we get
		$$
		\frac{dV}{da}=3\,a\,\frac{dA}{da}.
		$$
		
		From relation (\ref{1}) we have
		$$
		\frac{d A}{d a}=\frac{\sin a}{(1+2\cos a)\sqrt{(\cos a+1)(3\cos a+1)}}\,.
		$$
		
		Combining the latter two relations, by the Newton--Leibniz formula we express the volume of $\mathcal{T}$ in terms of its edge lengths
		\begin{equation*}\label{4}
		V=\int_0^a \frac{3\,t\,\sin t\;dt}{(1+2\cos t)\sqrt{(\cos t+1)\,(3\cos t+1)}}\,.
		\end{equation*}
	\end{proof}
	
	This result is similar to the formulae for the volume of a regular hyperbolic tetrahedron (see, e.g., \cite{AV2017}, Theorem~1).
	
	For $n=3$ and $a=c$ an antiprism $\mathcal{A}_n$ become a regular octahedron (see Fig.~3, right).
	
	\begin{theorem}\label{regoct}
		Let $\mathcal{O}(a)$ be a regular spherical octahedron with edge lengths $a$ and dihedral angles $A$. Then the volume $V=vol(\mathcal{O})$ can be found by either of the following formulae
		\begin{align*}\label{Volumeocta}
		V&=6\displaystyle \int_{\arccos(-\frac{1}{3})}^A \arccos\left(-\frac{\cos\varphi+1}{2\cos\varphi}\right)\,d\varphi,\\
		V&=\int_{0}^a \frac{6\,t\sin t\,dt}{(1+2\cos t)\sqrt{\cos t (1+\cos t)}}.
		\end{align*}
	\end{theorem}
	\begin{proof}
		Consider a regular spherical octahedron $\mathcal{O}(a)$. Its faces are equal regular triangles with side length $a$. Denote the face angles by $\alpha$ as shown on Fig.~\ref{soct1}, left. 
		\begin{figure}[htbp]
			\begin{center}
				\includegraphics[scale=0.3]{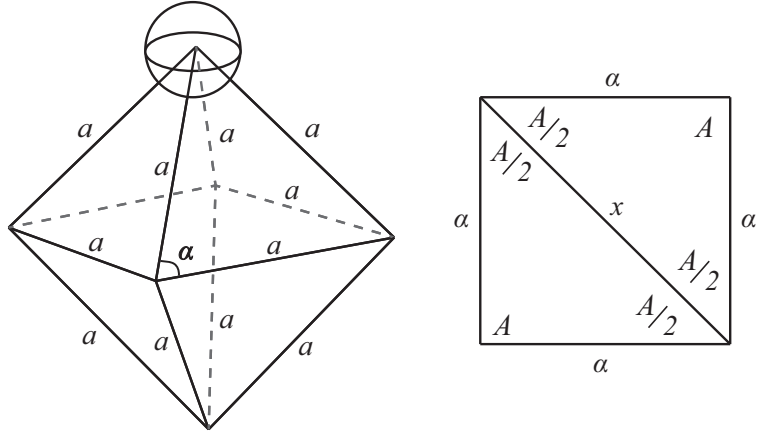}
			\end{center}
			\caption{Regular spherical octahedron $\mathcal{O}$ and its intersection with a sphere centered at a vertex}\label{soct1}
		\end{figure}
		
		By the cosine rule for spherical triangle we have
		$$
		\cos a={\cos}^2 a+{\sin}^2 a\;\cos\alpha\,,
		$$
		or equivalently
		\begin{equation}\label{equa1}
		\cos\alpha=\frac{\cos a}{1+\cos a}\,.
		\end{equation}
		
		Consider the intersection of the octahedron $\mathcal{O}$ by a sufficiently small sphere centred at a vertex of $\mathcal{O}$ (see Fig.~\ref{soct1}, right). The intersection is a regular spherical quadrilateral with equal angles $A$ and side lengths $\alpha$. Denote by $x$ the length of a diagonal of the quadrilateral. The diagonal divides the regular quadrilateral into two equal isosceles triangles. Consider the triangle, which has side lengths $\alpha,\alpha$ and  $x$. By the spherical sine rule we have
		\[
		\frac{\sin x}{\sin A}=\frac{\sin \alpha}{\sin\frac{A}{2}}
		\]
		or equivalently 
		\begin{equation}\label{equa2}
		\sin x = 2\sin\alpha \,\cos\frac{A}{2}.
		\end{equation}
		
		By the spherical cosine rule we get
		\[
		\cos x = \cos^2 \alpha + \sin^2 \alpha\,\cos A
		\]
		or equivalently 
		\begin{equation}\label{equa3}
		\cos x = 1-2\sin^2\alpha\,\sin\frac{A}{2}
		\end{equation}
		
		Excluding $x$ from the system of equations (\ref{equa2}) and (\ref{equa3}), we obtain
		\begin{equation}\label{equa4}
		\cos \alpha = \frac{1+\cos A}{1-\cos A}.
		\end{equation}
		
		Equating the right-hand sides in (\ref{equa1}) and (\ref{equa4}), we derive the relations between dihedral angles and edge lengths of $\mathcal{O}$
		\begin{align}
		\cos A& =- \frac{1}{1+2\cos a}\,,\label{equa5}\\
		a &={\arccos}\left(-\frac{\cos A+1}{2\,\cos A}\right).  \label{equa6}
		\end{align}
		
		Observe, that when $a\to 0$ the spherical octahedron $\mathcal{O}$ degenerates to a single point and its volume $V\to 0$. At the same time, its dihedral angles $\displaystyle A\to \arccos(-1/3)$.
		
		By the Schl\"afli formula (see, e.g., \cite{Vinberg}, Chap.~7, Sec.~2.2), we have the following expression for a differential of volume $V=vol(\mathcal{O})$
		\begin{equation*}\label{Schl2}
		dV=\sum_{\theta}\frac{\ell_{\theta}}{2}\,d\theta=6\,a\,dA,
		\end{equation*}
		where the sum is taken over all the edges of $\mathcal{O}$ and $\ell_{\theta}$ is the length of an edge with dihedral angle $\theta$ along it.
		
		We substitute (\ref{equa6}) in (\ref{Schl2}) and integrate it with respect to variable $A$ denoted by $\varphi$. Thus, by the Newton--Leibniz formula we express the volume of $\mathcal{O}$ in terms of its dihedral angles
		\begin{equation*}\label{vol1}
		V=6\int_{\arccos(-1/3)}^{A}{\arccos}\left(-\frac{\cos\varphi+1}{2\,\cos\varphi}\right)d\varphi\,.
		\end{equation*}
		
		Differentiating equation (\ref{Schl2}) with respect to $a$, we get
		$$
		\frac{dV}{da}=6\,a\,\frac{dA}{da}.
		$$
		
		From relation (\ref{equa5}) we have
		$$
		\frac{d A}{d a}=\frac{\sin a}{(1+2\cos a)\sqrt{\cos a(\cos a+1)}}.
		$$
		Combining the latter two relations, by the Newton--Leibniz formula we express the volume of $\mathcal{O}$ in terms of its edge lengths
		\begin{equation*}\label{vol2}
		V=\int_0^a \frac{6\,t\,\sin t\;dt}{(1+2\cos t)\sqrt{\cos t\,(\cos t+1)}}\,.
		\end{equation*}
	\end{proof}
	
	This result is similar to the formulae for the volume of a regular hyperbolic octahedron (see, e.g., \cite{AKM}, Sec.~4.2).
	
	\section{Existence conditions of a spherical antiprism}
	
	Consider a spherical antiprism $\mathcal{A}_n(a,c)$ given by its spherical edge lengths $a,c$. We denote by $a$ the length of those edges that form top and bottom $n$-gonal faces. We set $c$ for the length of the lateral edges.
	
	\begin{theorem}\label{existsant}
		A spherical antiprism $\mathcal{A}_n(a,c)$ with symmetry group $S_{2n}$ is exist if and only if the following conditions hold
		\begin{align*}
		1+ \cos a - 2\left(1+\cos\frac{\pi}{n}\right)\cos c+2 \cos  \frac{\pi}{n}&\geq 0,\\
		1+ \cos a + 2\left(1-\cos\frac{\pi}{n}\right)\cos c-2 \cos  \frac{\pi}{n}&\geq 0,\\
		\cos a - \cos \frac{2\pi}{n}&\geq0.
		\end{align*}
	\end{theorem}
	\begin{proof}
		Since the antiprism has a symmetry group $S_{2n}$, it is inscribed in a cylinder. This cylinder is formed by parallel translation of the circumscribed circle of the upper $n$-gonal face of the antiprism along the axis of symmetry at a distance $d$ between the centers of the upper and lower $n$-gonal bases. Note that the space $\mathbb{S}^3$ is compact and convex. Thus, for the existence of a spherical antiprism, it is necessary and sufficient that there exist two regular $n$-gonal bases and the distance $d=d(a,c)$ between their centers can be correctly determined. Equivalently, the following conditions must be hold
		\begin{equation}\label{e}
		\begin{split}
		0&\leq n\,a\leq 2\pi,\\
		0&\leq d\leq 2\pi.
		\end{split} 
		\end{equation}
		
		Consider limiting cases when some equalities are attained in (\ref{e}). $n\,a=0$ when antiprism $\mathcal{A}_n(a,c)$ degenerates into a spherical line segment (i.e. a segment of a great circle in $\mathbb{S}^3$). In particular, $d=2\pi$ when this segment is the whole great circle. $d=0$ when $\mathcal{A}_n(a,c)$ degenerates into a spherical $2n$-gon. In particular, $n\,a=2\pi$ when this $2n$-gon lies on a great circle in $\mathbb{S}^3$.
		
		Let us determine the value $\cos d$ as a function of $\cos a$ and $\cos c$. We connect the center $C_1$ and the vertex $v_1$ of the lower $n$-gonal face with a segment of length $R$. Then in the corresponding lateral triangle we connect the vertex $v_1$ and the middle $M$ of the edge $v_2v_3$ of the upper $n$-gonal face by by the segment $H$. We denote the segment connecting the midpoint $M$ and the center $C_2$ of the upper $n$-gonal face by $h$ (see Fig.~\ref{sant_1}, left).
		\begin{figure}[h]
			\begin{center}
				\includegraphics[scale=0.6]{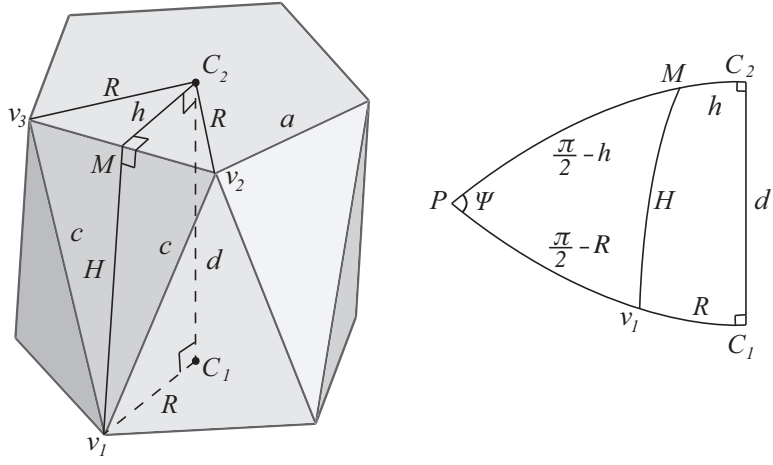}
			\end{center}
			\caption{Quadrilateral $C_1VMC_2$}\label{sant_1}
		\end{figure}
		
		Consider the triangle $v_1Mv_2$, by the spherical Pythagorean theorem we have
		\begin{equation}\label{cosH}
		\cos H =\frac{\cos c}{\cos\frac{a}{2}}.
		\end{equation}
		
		Consider the triangle $v_2C_2v_3$, by the spherical cosine rule we have
		\begin{equation*}
		\cos^2 R =\frac{\cos a-\cos\frac{2\pi}{n}}{1-\cos\frac{2\pi}{n}}.
		\end{equation*}
		
		Since the circumscribed circle with the radius R of the $n$-gonal base of the antiprism cannot be larger than the great circle of the spherical space, the inequalities $a\leq2\pi/n$ and $R\leq\pi/2$ hold. Therefore
		\begin{equation}\label{cosR}
		\cos R =\sqrt{\frac{\cos a-\cos\frac{2\pi}{n}}{1-\cos\frac{2\pi}{n}}}.
		\end{equation}
		
		In the triangle $C_2Mv_2$, by the spherical Pythagorean theorem we have
		\begin{equation}\label{cosh}
		\cos h=\frac{\cos R}{\cos\frac{a}{2}}=\frac{1}{\cos\frac{a}{2}}\sqrt{\frac{\cos a-\cos\frac{2\pi}{n}}{1-\cos\frac{2\pi}{n}}}.
		\end{equation}
		
		Geodesic curves $C_1v_1$ and $C_2M$ meet at pole $P$ (see Fig.~\ref{sant_1}, right). Note that $C_1P=C_2P=\pi/2$. Consider the triangle $C_1PC_2$, by the spherical cosine rule we get $\cos\Psi=\cos d$, where $\Psi$ is the angle at pole $P$. Further, in the triangle $PMv_1$ by the spherical cosine rule we have
		\begin{equation*}
		\cos d=\cos\Psi=\frac{\cos H -\cos (\frac{\pi}{2}-h)\cos(\frac{\pi}{2}-R)}{\sin(\frac{\pi}{2}-h)\sin(\frac{\pi}{2}-R)}=\frac{\cos H-\sin h\,\sin R}{\cos h\,\cos R}.
		\end{equation*}
		
		We substitute (\ref{cosH}), (\ref{cosR}) and (\ref{cosh}) in the latter equation, then we obtain 
		\begin{equation}\label{cosd}
		\cos d = \frac{2\cos c \,(1-\cos^2\frac{\pi}{n})-\cos \frac{\pi}{n}\,(1-\cos a)}{\cos a - \cos \frac{2\pi}{n}}.
		\end{equation}
		
		The first inequality $0\leq n\,a\leq 2\pi$ in (\ref{e}) is equivalent to 
		\begin{equation*}
		\cos a - \cos \frac{2\pi}{n}\geq0.
		\end{equation*}
		
		The second inequality $0\leq d\leq 2\pi$ in (\ref{e}) is equivalent to 
		\begin{equation*}
		-1\leq\cos d\leq 1.
		\end{equation*}
		
		We substitute (\ref{cosd}) in the last double inequality. Then by straightforward calculation we obtain
		\begin{align*}
		1+ \cos a - 2\left(1+\cos\frac{\pi}{n}\right)\cos c+2 \cos  \frac{\pi}{n}&\geq 0,\\
		1+ \cos a + 2\left(1-\cos\frac{\pi}{n}\right)\cos c-2 \cos  \frac{\pi}{n}&\geq 0.
		\end{align*}
	\end{proof}
	
	\section{Isosceles spherical trapezoid}
	
	An {\em isosceles spherical trapezoid} is a spherical quadrilateral that admits a mirror symmetry about the line passing through the midpoints of a pair of its opposite sides (see Fig.~\ref{tra1}).
	\begin{figure}[htbp]
		\begin{center}
			\includegraphics[scale=0.6]{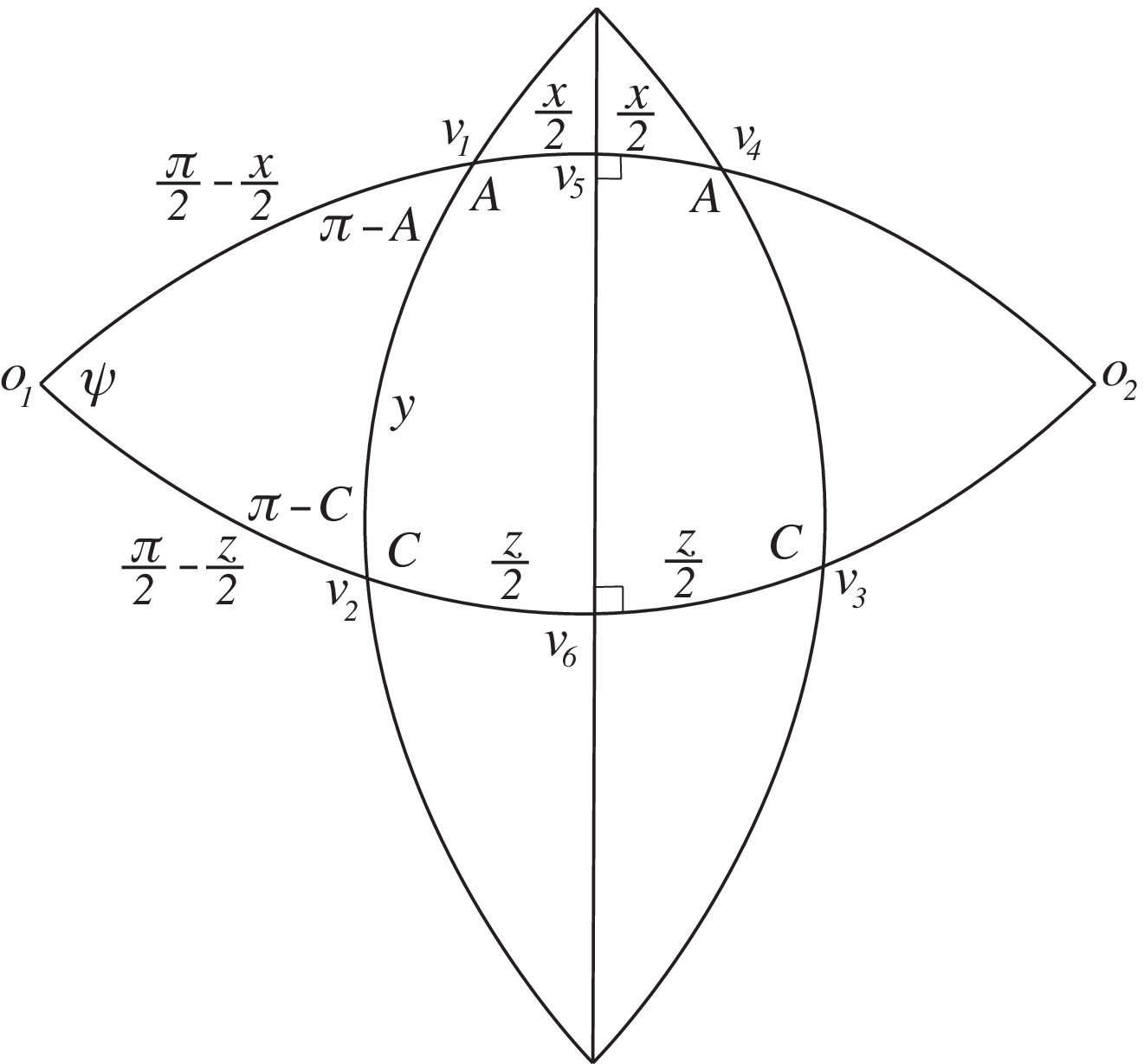}
		\end{center}
		\caption{Isosceles spherical trapezoid $v_1v_2v_3v_4$}\label{tra1}
	\end{figure}
	
	For further consideration we will need the following property of an isosceles spherical trapezoid. This results might be known.
	\begin{theorem}\label{formula_tra}
		Let $\tau$ be an isosceles spherical trapezoid with bases $x, z$ and lateral sides $y$. Let $A$ be the angles at base $x$ and $C$ be the angles at base $z$. Then the following relations hold
		\begin{align*}
		\cos A&=\frac{\cos y\,\sin\frac{x}{2}-\sin\frac{z}{2}}{\sin y\,\cos\frac{x}{2}},\\
		\cos C&=\frac{\cos y\,\sin\frac{z}{2}-\sin\frac{x}{2}}{\sin y\,\cos\frac{z}{2}}.
		\end{align*}
	\end{theorem}
	\begin{proof}
		Consider an isosceles spherical trapezoid, denote by $o_1, o_2$ two antipodal points of intersection of the sides $x$ and $z$ (see Fig.~\ref{tra1}). The axis of symmetry passes through the midpoints $v_5, v_6$ of the sides $x, z$, respectively. Note that the points $v_5, v_6$ are also the midpoints of two semicircles of geodesic curves connecting two points $o_1$ and $o_2$. Hence, the equalities $o_1v_5=o_1v_6= \pi/2$ are hold.
		
		Consider triangle $o_1v_2v_2$, by the spherical cosine rule we have
		\begin{align*}
		\cos (\pi-A)&=\frac{\cos (\frac{\pi}{2}-\frac{z}{2})- \cos y\,\cos (\frac{\pi}{2}-\frac{x}{2})}{\sin y\,\sin (\frac{\pi}{2}-\frac{x}{2})},\\
		\cos (\pi-C)&=\frac{\cos (\frac{\pi}{2}-\frac{x}{2})- \cos y\,\cos (\frac{\pi}{2}-\frac{z}{2})}{\sin y\,\sin (\frac{\pi}{2}-\frac{z}{2})},
		\end{align*}
		or equivalently
		\begin{align*}
		\cos A&=\frac{\cos y\,\sin\frac{x}{2} -\sin\frac{z}{2}}{\sin y \,\cos\frac{x}{2}},\\
		\cos C&=\frac{\cos y\,\sin\frac{z}{2} -\sin\frac{x}{2}}{\sin y \,\cos\frac{z}{2}}.
		\end{align*}
	\end{proof}
	
	\section{Dihedral angles of a spherical antiprism}
	
	\begin{theorem}\label{sant2}
		Let $\mathcal{A}_n(a,c)$ be a spherical antiprism with $2n$ vertices given by the edge lengths $a,c$. Then the dihedral angles $A$ and $C$ along edges of $\mathcal{A}_n(a,c)$ can be found by the formulae
		\begin{equation}\begin{split}\label{sant3}
		\cos A&=\frac{\sqrt{1-\cos a}\,(2\cos c\,\cos\frac{\pi}{n} - \cos a -1)}{\sqrt{2(1+\cos a - 2 \cos^2 c)(\cos a - \cos \frac{2\pi}{n})}},\\
		\cos C&=\frac{\cos c -\cos a\,\cos c +2(\cos^2 c-1)\cos \frac{\pi}{n}}{1+\cos a - 2 \cos^2 c}.
		\end{split}\end{equation}
	\end{theorem}
	\begin{proof}
		Consider a spherical antiprism $\mathcal{A}_n(a,c)$ given by its edge lengths $a,c$. Denote by $A,C$ the dihedral angles along the edges $a,c$, respectively. The upper face of $\mathcal{A}_n(a,c)$ is a regular spherical $n$-gon with edges $a$. Consider an isosceles triangle, the vertices of which are the center and two adjacent vertices of the upper face. Note that the angle of this triangle at the vertex in the center of the upper face is equal to $2 \pi /n$, and the remaining angles are equal to half of the angle $x$ of a regular spherical $n$-gon (see Fig.~\ref{sant1}, left).
		\begin{figure}[h]
			\begin{center}
				\includegraphics[scale=0.6]{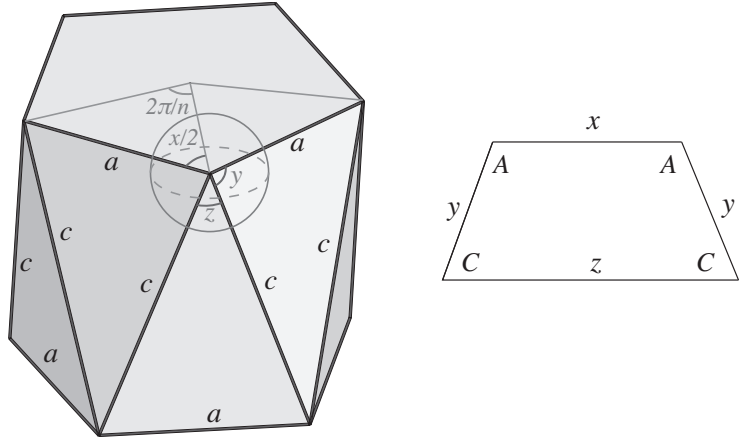}
			\end{center}
			\caption{Spherical antiprism $\mathcal{A}_5$ and its intersection with a sphere centered at a vertex}\label{sant1}
		\end{figure}
		
		By the spherical cosine rule we have
		\begin{equation*}
		\cos a = \frac{\cos\frac{2\pi}{n}+\cos^2\frac{x}{2}}{\sin^2\frac{x}{2}}
		\end{equation*}
		or equivalently
		\begin{equation*}
		\cos^2\frac{x}{2} = \frac{\cos a - \cos\frac{2\pi}{n}}{\cos a +1}\,.
		\end{equation*}
		
		Since  $x\leq\pi$, we have
		\begin{equation}\begin{split}\label{cos-sin-x/2}
		\cos\frac{x}{2}&=\sqrt{ \frac{\cos a - \cos\frac{2\pi}{n}}{\cos a +1}},\\
		\sin\frac{x}{2}&=\sqrt{\frac{\cos\frac{2\pi}{n}+1}{\cos a+1}}.
		\end{split}\end{equation}
		
		Consider one of the lateral faces of $\mathcal{A}_n(a,c)$, which is an isosceles triangle with the side lengths $a, c, c$. Denote its angles by $z,y,y$, as shown in the figure \ref{sant1}. By the spherical cosine rule we get
		\begin{align}\label{cos-y}
		\cos y&=\frac{\cos c - \cos c\,\cos a}{\sin c\,\sin a}=\cot c\,\tan\frac{a}{2},\\
		\cos z&= \frac{\cos a- \cos^2 c}{\sin^2 c}.\nonumber
		\end{align}
		
		Since $\cos z = 1-2 \sin^2\frac{z}{2} = 2\cos^2\frac{z}{2}-1$ and $y,z,c\leq\pi $, we have
		\begin{equation}\begin{split}\label{sin-cos-z/2-sin-y}
		&\sin\frac{z}{2} = \frac{\sqrt{1-\cos a}}{\sqrt{2}\sin c} ,\\
		&\cos\frac{z}{2}= \frac{\sqrt{\cos a- \cos 2c}}{\sqrt{2}\sin c},\\
		&\sin y = \sqrt{1-\cot^2 c\,\tan^2 \frac{a}{2}}.
		\end{split}\end{equation}
		
		Consider an intersection of the spherical antiprism $\mathcal{A}_n(a,c)$ with a sphere of sufficiently small radius centered at any vertex (see Fig.~\ref{sant1}).
		In the intersection, we have an isosceles spherical trapezoid with angles $A, C$ and sides $x,y,z,y$. Here $A,C$ are equal to the dihedral angles of $\mathcal{A}_n(a,c)$, and $x,y,z$ are the face angles of the triangles considered above (see Fig.~\ref{sant1}). By Theorem~\ref{formula_tra} for the isosceles spherical trapezoid we have
		\begin{align*}
		\cos A&=\frac{\cos y\,\sin\frac{x}{2}-\sin\frac{z}{2}}{\sin y\,\cos\frac{x}{2}},\\
		\cos C&=\frac{\cos y\,\sin\frac{z}{2}-\sin\frac{x}{2}}{\sin y\,\cos\frac{z}{2}}.
		\end{align*}
		
		We substitute (\ref{cos-sin-x/2}), (\ref{cos-y}) and (\ref{sin-cos-z/2-sin-y}) in the last two equations. Hence, we obtain the relations between the dihedral angles and the edge lengths of the antiprism
		\begin{align*}
		\cos A&=\frac{2\sin\frac{a}{2}(\cos c\,\cos\frac{\pi}{n} - \cos^2\frac{a}{2})}{\sqrt{(\cos a - \cos 2c)(\cos a - \cos\frac{2\pi}{n})}},\\
		\cos C&=\frac{2(\cos c\,\sin^2\frac{a}{2}-\sin^2 c\,\cos\frac{\pi}{n})}{\cos a- \cos 2c}
		\end{align*}
		or equivalently
		\begin{align*}
		\cos A&=\frac{\sqrt{1-\cos a}\,(2\cos c\,\cos\frac{\pi}{n} - \cos a -1)}{\sqrt{2(1+\cos a - 2 \cos^2 c)(\cos a- \cos\frac{2\pi}{n})}},\\ 
		\cos C&=\frac{\cos c -\cos a\,\cos c +2(\cos^2 c-1)\cos\frac{\pi}{n}}{1+\cos a - 2 \cos^2 c}.
		\end{align*}
	\end{proof}

For $n=2$ and $a=c$, antiprism $\mathcal{A}_n(a,c)$ is a regular tetrahedron $\mathcal{T}(a)$. 
From Theorem~\ref{sant2} we directly obtain the following.
\begin{corollary}
	Let $\mathcal{T}(a)$ be a regular spherical tetrahedron given by the edge length $a$. Then the dihedral angles $A$ of $\mathcal{T}(a)$ can be found by the formula
	\begin{equation*}
	\cos A=\frac{\cos a}{1+2\cos a}.
\end{equation*}
\end{corollary}	
This equation coincides with formula~(\ref{1}).

For $n=3$ and $a=c$, antiprism $\mathcal{A}_n(a,c)$ is a regular octahedron $\mathcal{O}(a)$. 
From Theorem~\ref{sant2} we have the following.
\begin{corollary}
	Let $\mathcal{O}(a)$ be a regular spherical octahedron given by the edge length $a$. Then the dihedral angles $A$ of $\mathcal{O}(a)$ can be found by the formula
	\begin{equation*}
	\cos A=-\frac{1}{1+2\cos a}.
	\end{equation*}
\end{corollary}	
The latter equation coincides with formula~(\ref{equa5}).
	
	\section{Volume formula for a spherical antiprism}
	
	\begin{theorem}\label{volumetheo}
		Let $\mathcal{A}_n(a,c)$ be a spherical antiprism with $2n$ vertices given by the edge lengths $a,c$. Then the volume  $V=vol(\mathcal{A}_n(a,c))$ can be found by the formula
		\begin{equation}\label{sant9}
		V=n \int_{c_0}^c \frac{a\,G+c\,H}{(1+\cos a-2\cos^2 c)\sqrt{R}}\,dt,
		\end{equation}
		where
		\begin{flalign*}
		G&=-2\left(\cos t-\cos\frac{\pi}{n}\right)\sin a\,\sin t,\\
		H&=(1-\cos a)\left(1+\cos a+2\cos^2 t-4\cos t\,\cos\frac{\pi}{n}\right),\\
		R&=(1+\cos a)^2-4\left(\cos^2 c+\sin^2 c\,\cos^2\frac{\pi}{n}-(1-\cos a)\cos c\,\cos\frac{\pi}{n}\right).
		\end{flalign*}
		and $c_0$ is the root of the equation $2\cos c\left(1+\cos\frac{\pi}{n}\right)=1+\cos a+2\cos\frac{\pi}{n}$.
	\end{theorem}
	\begin{proof}
		Consider a spherical antiprism $\mathcal{A}_n(a,c)$ given by the edge lengths $a,c$.  
		By definition, $\mathcal{A}_n(a,c)$ has the symmetry group $S_{2n}$. According to Theorem~\ref{existsant}, the region of existence of such an antiprism has the form
		\begin{multline*}
		\Omega=\{(\cos a,\cos c)\,:\,1+ \cos a - 2\left(1+\cos\frac{\pi}{n}\right)\cos c+2 \cos  \frac{\pi}{n}\geq 0,\\
		1+ \cos a + 2\left(1-\cos\frac{\pi}{n}\right)\cos c-2 \cos  \frac{\pi}{n}\geq 0,\;
		\cos a - \cos \frac{2\pi}{n}\geq0\},
		\end{multline*}
		in the system of coordinates $\cos a,\cos c$ (see Fig.~\ref{sant4}).
		\begin{figure}[h]
			\centerline{\includegraphics[width=4.9in]{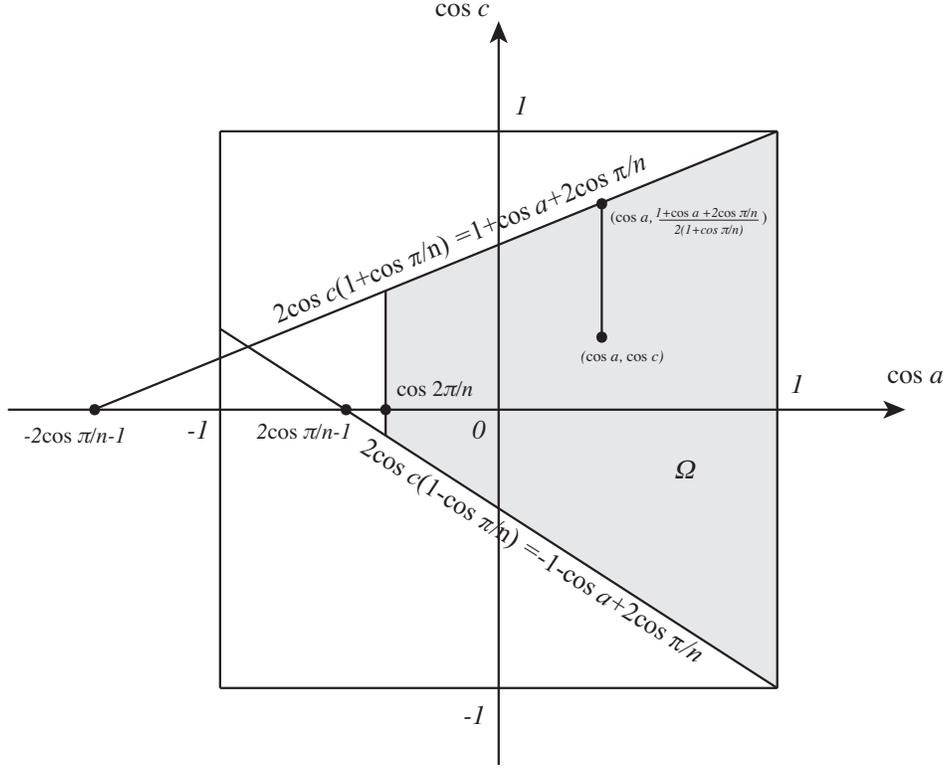}}
			\caption{Region of existence of the spherical antiprism $\mathcal{A}_n(a,c)$.\label{sant4}}
		\end{figure}
		
		The boundary of region $\Omega$ consists of four lines segments  
		\begin{align*}
		\cos a&=1,\\ 
		\cos a&= \cos\frac{2\pi}{n},\\
		2\cos c\,(1+\cos \pi/n)&=1+\cos a +2\cos\frac{\pi}{n},\\
		2\cos c\,(1-\cos \pi/n)&=-1-\cos a +2\cos\frac{\pi}{n}.
		\end{align*}
		
		From the proof of Theorem~\ref{existsant} we know that antiprism $\mathcal{A}_n(a,c)$ loses its dimension on the boundary of $\Omega$, degenerating into a spherical line segment or a spherical $2n$-gon (for more details, see the explanation when equalities attain in (\ref{e})). Therefore, the equality $V=0$ holds on the boundary of $\Omega$.
		
		Denote by $A,C$ the dihedral angles along the edges $a,c$ of antiprism $\mathcal{A}_n(a,c)$, respectively. According to Theorem~\ref{sant2}, the dihedral angles are uniquely determined by the edge lengths. Differentiating the volume as a composite function of the edge lengths, we get
		\begin{equation}\label{sant5}
		\frac{\partial V}{\partial c}=\frac{\partial V}{\partial A}\frac{\partial A}{\partial c}+\frac{\partial V}{\partial C}\frac{\partial C}{\partial c}.
		\end{equation}
		
		By the Schl\"afli formula we have
		\begin{equation*}
		d V=\sum_{\theta}\frac{\ell_{\theta}}{2}\,d\theta=n\,a\,dA\,+n\,c\,dC,
		\end{equation*}
		where the sum is taken over all edges of $\mathcal{A}_n(a,c), \;\ell_{\theta}$ denotes the length of an edge, and $\theta$ is the dihedral angle along it. Hence,
		\begin{equation}\label{sant6}
		\dfrac{\partial V}{\partial A}=n\,a,\quad\dfrac{\partial V}{\partial C}=n\,c.
		\end{equation} 
		
		From (\ref{sant3}) by straightforward calculation we obtain
		\begin{equation}\begin{split}\label{derivativesant}
		\frac{\partial A}{\partial c}&=-\frac{2\left(\cos c-\cos\frac{\pi}{n}\right)\sin a\;\sin c}{(1+\cos a-2\cos^2 c)\sqrt{R}},\\
		\frac{\partial C}{\partial c}&=\frac{(1-\cos a)(1+\cos a+2\cos^2 c-4\cos c\cos\frac{\pi}{n})}{(1+\cos a-2\cos^2 c)\sqrt{R}},
		\end{split}\end{equation} 
		where
		$R=(1+\cos a)^2-4\left(\cos^2 c+\sin^2 c\,\cos^2\frac{\pi}{n}-(1-\cos a)\cos c\,\cos\frac{\pi}{n}\right).$
		
		We substitute the expressions (\ref{sant6}) and (\ref{derivativesant}) in (\ref{sant5}). Then we get
		\begin{equation}\label{sant7}
		\frac{\partial V}{\partial c}=n\frac{a\,G+c\,H}{(1+\cos a-2\cos^2 c)\sqrt{R}},
		\end{equation}
		where
		\begin{align*}
		G&=-2\left(\cos c-\cos\frac{\pi}{n}\right)\sin a\;\sin c,&\\
		H&=(1-\cos a)\left(1+\cos a+2\cos^2 c-4\cos c\cos\frac{\pi}{n}\right).&
		\end{align*}
		
		An integral of the differential form 
		\begin{equation}\label{sant8}
		d V=\dfrac{\partial V}{\partial a}da+\dfrac{\partial V}{\partial c}dc
		\end{equation} 
		does not depend on the integration path connecting two fixed points in $\Omega$. Since the volume vanishes on the boundary of $\Omega$ then by the Newton--Leibniz formula, the volume is equal to the integral of the form (\ref{sant8}) along any piecewise smooth path $\gamma\subset\Omega$ beginning from any point at the boundary of $\Omega$ with end at the point $(\cos a,\cos c)$. We substitute the expression (\ref{sant7}) in (\ref{sant8}). Then we integrate the differential form (\ref{sant8}) over the vertical segment shown in Fig.~\ref{sant4} which connects the boundary of $\Omega$ with the point $(\cos a,\cos c)$. Thus, we arrive at the formula (\ref{sant9}). To distinguish the edge length $c$ from the variable of integration, we denote this variable by $t$.
	\end{proof}

\end{document}